    \documentclass[onefignum,onetabnum]{siamart220329}
    
    \usepackage{algorithm}
    \usepackage{algpseudocode}
    \usepackage{color} 
    \usepackage{booktabs}       % professional-quality tables
    \usepackage{amsfonts}       % blackboard math symbols
    \usepackage{nicefrac}       % compact symbols for 1/2, etc.
    \usepackage{microtype}      % microtypography
    \usepackage{amsmath, amssymb}		% Can be removed after putting your text content
    \usepackage{graphicx}
    \usepackage[numbers]{natbib}
    \usepackage{hyperref} 
    \usepackage{nameref} 
    \usepackage{subcaption}

    \newcommand{\inner}[1]{\left<#1\right>}
    
    \title{ 
    Optimizing Mixed Quantum Channels via Projected Gradient Dynamics
    }
    \author{
    Matthew M. Lin\thanks{
    Department of Mathematics, National Cheng Kung University, Tainan 701, Taiwan ({mhlin@mail.ncku.edu.tw})} \and 
    Bing-Ze Lu\thanks{Oden Institute for Computational Engineering and Sciences, The University of Texas at Austin, TX 78712. (bingzelu@utexas.edu).} 
    }

\begin{document}
    
    \maketitle
    
    % REQUIRED
    \begin{abstract}
    Designing a mixed quantum channel is challenging due to the complexity of the transformations and the probabilistic mixtures of more straightforward channels involved. Fully characterizing a quantum channel generally requires preparing a complete set of input states, such as a basis for the state space, and measuring the corresponding output states. In this work, we begin by investigating a single input-output pair using projected gradient dynamics. This approach applies optimization flows constrained to the Stiefel manifold and the probabilistic simplex to identify the original quantum channel. The convergence of the flow is guaranteed by its relationship to the Zariski topology. We present numerical investigations of models adapted to various scenarios, including those with multiple input-output pairs, highlighting the flexibility and efficiency of our proposed method.
    
    \end{abstract}
    
    % REQUIRED
    
    % REQUIRED
    
    \section{Introduction}
    Recent advances in quantum simulators and processors have significantly improved hardware capabilities and measurement techniques. However, fully characterizing quantum dynamics or channels remains a fundamental challenge. To address this, quantum process tomography (QPT), also known as channel identification, provides a systematic framework for reconstructing an unknown quantum process from experimental data. By determining how quantum systems evolve in response to various input and output states, QPT allows us to mathematically describe the process via the notation $\Phi$, which maps an input state $\rho$ to an output state \(\sigma\):  
    \[
    \sigma := \Phi(\rho).
    \]  
    
    In this work, we utilize prior knowledge of input and output states to recover a mixed quantum channel, if it exists, by solving the following optimization problem:  
    \begin{subequations}\label{eq:optimization}
    \begin{align}
    \text{Minimize} & \quad \frac{1}{2} \| \sigma - \sum_{k=1}^r p_k U_k \rho U_k^* \|_F^2, \\
    \text{subject to} & \quad U_k \in \mathcal{S}_n, \quad k = 1, \ldots, r, \\
    & \quad \mathbf{p} \in \Delta^{r-1},
    \end{align}
    \end{subequations}
    where $\mathcal{S}_n$ denotes the set of $n$-by-$n$ unitary matrices, $\Delta^{r-1}$ represents the probability simplex, denoted as:
    \begin{equation}\label{eq:simplex}
    \Delta^{r-1} := \left\{ \mathbf{p} \in \mathbb{R}^r \mid \mathbf{p} = [p_k] \geq 0, \sum_{k=1}^r p_k = 1 \right\},
    \end{equation}
    and $\|\cdot\|_F$ is the Frobenius norm.  
    
    This formulation, which we refer to as the optimization of mixtures of unitary operations, is a fundamental problem in quantum computing and quantum information theory. It has broad applications, including quantum channel approximation, quantum state synthesis, and noise modeling~\cite{Nielsen2010, Watrous2018,Brun2020}. For example, consider the depolarizing channel:  
    \begin{equation}\label{eq:depolar}
    \Phi(\rho) := (1-p) \rho + \frac{p}{3}\left(X \rho X + Y \rho Y + Z \rho Z\right),
    \end{equation}
    where $p$ is the probability of depolarization and $X$, $Y$, $Z$ are Pauli matrices. Without knowing $p$ or the underlying unitary operations a priori, our method aims to approximate $\Phi$ using a convex mixture of unitary operations.  
    
    On the other hand, the discrepancy metric utilized in~\eqref{eq:optimization} is grounded in different operational paradigms and mathematical formulations. Rather than using the Frobenius norm,  a more common measure is~\emph{fidelity}, defined as:  
    \[
    F(\sigma, {\rho}) = \left( \operatorname{Tr} \big( \sqrt{\sqrt{\sigma} {\rho} \sqrt{\sigma}} \big) \right)^2,
    \]  
    to assess the similarity between two quantum states $\sigma$ and $\rho$. However, directly optimizing fidelity can be challenging due to the complexity introduced by the square-root and trace operations. To avoid these difficulties, we adopt the Frobenius norm as a computationally efficient alternative in~\eqref{eq:optimization}. Unlike fidelity, the Frobenius norm simplifies optimization by avoiding complex matrix operations such as nested square roots and matrix traces. Although minimizing the Frobenius norm does not explicitly maximize fidelity, empirical results suggest that it often yields high-fidelity approximations, mainly when the target state and the approximate state are sufficiently close. This makes it a practical option for rebuilding quantum channels while ensuring computational efficiency.
    
    Specifically, this work employs a gradient flow-based method that actively refines the required number of unitary operations. Note that this value $r$ in~\eqref{eq:optimization} quantifies the complexity of decomposition and is essential to characterize the minimal resources required for tasks such as the quantum channel approximation and state synthesis. To solve the optimal problem~\eqref{eq:optimization}, one crucial aspect is to determine the minimum number $r$ of unitary operations needed for an accurate decomposition. A similar but more theoretical discussion of the minimal decomposition of quantum channels is provided by Lancien and Winter~\cite{Lancien2024}, who examine the approximation of quantum channels through completely positive maps with low Kraus rank, providing insights into how these decompositions can be optimized to minimize operational complexity. 
    % Similarly, Magesan in ~\cite{Magesan2010} examines the depolarizing behavior of quantum channels in higher dimensions and their implications for representation through unitary operations. 
    Beginning with a higher rank $r$, we demonstrate how to dynamically utilize the gradient flow method to adjust this parameter during the computation process. This approach seeks to make the approximation of a quantum channel computationally feasible using the fewest possible unitary operations.

    The remainder of this paper is organized as follows. In Section 2, we present the application of the projected gradient flow to solve the optimization problem. In Section 3, we provide a thorough analysis of the proposed method, proving that the objective function consistently decreases as expected, and establishing key convergence results. In Section 4, we validate our theoretical insight through numerical experiments, including a practical application centered on recovering the depolarizing quantum channel~\eqref{eq:depolar}. Finally, Section 5 offers concluding remarks.

    \section{Gradient flows}
     Building upon the standard Euclidean algorithm~\cite{Absil2008, Gao2020, Boumal2022}, we introduce a continuous-time flow to solve~\eqref{eq:optimization}. The main advantage of this algorithm is that the individual iterate also stays on the given constraints and simultaneously yields the optimal solution once it converges. Given a fixed rank $r$, the nearest problem given in~\eqref{eq:optimization} is to find a probability parameter $\mathbf{p} = (p_1,\ldots,p_r)$ and unitary matrices $U_k\in\mathbb{C}^{n\times n}$, $k= 1,\ldots, r$, such that the objective function 
    \begin{equation}\label{eq:objfun}
    f(\mathbf{p},U_1,\ldots,U_k) := \dfrac{1}{2}\|\sigma - \sum_{k=1}^r p_k U_k \rho U_k^* \|_F^2
    \end{equation}
    is minimized. The function $f$ in~\eqref{eq:objfun} is not analytic unless it is a zero function. 
    Although the function $f$ is not holomorphic (i.e., complex differentiable), we can still compute its derivatives concerning the real and imaginary parts of each variable. To proceed, let $U^\mathfrak{R}$ and $U^\mathfrak{I}$ denote the real and imaginary parts of the complex matrix $U$, respectively. Using the concept of Wirtinger derivatives, the following result provides explicit expressions for the components of the derivatives of $f$ concerning the real variables.
    
    \begin{theorem}\label{thm:gradient}
    For $k = 1, \ldots, t$, the components of the derivative of $f$ with respect to the real and imaginary parts,  $ 
    U_k^\mathfrak{R}
    $  and $ 
    U_k^\mathfrak{I}
    $,  
    of $U_k$ and $p_k$ are given as follows:
    \begin{equation}\label{eq:par2}
    \left\{
        \begin{array}{ccl}
            \dfrac{\partial f}{\partial U_k^\mathfrak{R}} &=&  2 \mathfrak{Re}\left(p_k\left( \mathcal{A}_k - \sigma \right) U_k \rho\right), \\[10pt]
            \dfrac{\partial f}{\partial U_k^\mathfrak{I}} &=& 2 \mathfrak{Im}\left(p_k\left(\mathcal{A}_k - \sigma \right) U_k \rho\right), \\[10pt]
            \dfrac{\partial f}{\partial p_k} &=& 
            \mathfrak{Re}\left(\inner{\mathcal{A}_k - \sigma, U_k \rho U_k^*}\right)+p_k \|\rho\|_F^2,
        \end{array}
    \right.
\end{equation}

    where $\mathcal{A}_k$ is defined as:
    \begin{equation*}
        \mathcal{A}_k := \sum_{j \neq k} p_j U_j \rho U_j^*.
    \end{equation*}
    \end{theorem}
    \begin{proof}
    The objective function $g$ can be equivalently expressed as:
    \begin{eqnarray}\label{eq:objdec}
        f(\mathbf{p}, U_1, \ldots, U_k) &=& 
        \frac{1}{2} \left( \langle \sigma - \mathcal{A}_k, -p_k U_k \rho \overline{U_k}^\top \rangle \right. \nonumber \\
        && \left. + \langle -p_k U_k \rho \overline{U_k}^\top, \sigma - \mathcal{A}_k \rangle + \|\sigma - \mathcal{A}_k\|_F^2 + p_k^2\|\rho\|_F^2 \right),
    \end{eqnarray}
    where the inner product is defined as:
    \begin{equation*}
        \langle X, Y \rangle := \sum_{i,j=1}^n x_{ij} \overline{y_{ij}}
    \end{equation*}
    for $X, Y \in \mathbb{C}^{n \times n}$.  Correspondingly, we let the real-valued inner product over the real field is
    \begin{equation*}
        \langle X, Y \rangle_{\mathbb{R}} := \sum_{i,j=1}^n x_{ij} y_{ij}.
    \end{equation*}
    
    From direct computation, the derivatives of $g$ are obtained as
    \begin{eqnarray*}
        \frac{\partial f}{\partial U_k} \cdot \Delta U &=& \langle p_k \overline{(\mathcal{A}_k - \sigma) U_k \rho}, \Delta U \rangle_{\mathbb{R}}, \\
        \frac{\partial f}{\partial \overline{U_k}} \cdot \Delta U &=& \langle p_k (\mathcal{A}_k - \sigma) U_k \rho, \Delta U \rangle_{\mathbb{R}}, \\
        \frac{\partial f}{\partial p_k} &=& \mathfrak{Re} \left(\langle \mathcal{A}_k - \sigma, U_k \rho \overline{U_k}^\top \rangle \right)+p_k \|\rho\|_F^2.
    \end{eqnarray*}
    
    Using the properties of Wirtinger derivatives~\cite{Brandwood1983, Krantz2001}, the partial derivatives of $g$ with respect to $ 
    U_k^\mathfrak{R}
    $ and $U_k^\mathfrak{I}
    $ are derived as
    \begin{eqnarray*}
        \dfrac{\partial f}{\partial U_k^\mathfrak{R}
        } &=& \frac{\partial f}{\partial U_k} + \frac{\partial f}{\partial \overline{U_k}}, \\
            \dfrac{\partial f}{\partial U_k^\mathfrak{I} } &=& \imath\left(\frac{\partial f}{\partial U_k} - \frac{\partial f}{\partial \overline{U_k}}\right), 
    \end{eqnarray*}
    which yields the result in \eqref{eq:par2}. This completes the proof.
    \end{proof}

    Theorem~\ref{thm:gradient} establishes the fundamental derivative information necessary to construct a descent flow for the optimization process. However, the optimization problem posed in~\eqref{eq:optimization} is a constrained optimization problem in which the descent flow must respect the constraints imposed. Specifically, the flow must be limited to the feasible region defined by the constraints. To explore this in greater depth, let $\gamma(t) = (\mathbf{p}(t), U_{1}(t), \ldots, U_{r}(t)$ represent a smooth curve within the domain $ \Delta^{r-1} \times \mathcal{S}_{n_1} \times \cdots \times \mathcal{S}_{n_r} $, where $t \in \mathbb{R}$, and assume that $ \gamma(0) = (\mathbf{p}, U_{1}, \ldots, U_{r}) $ lies in the interior of this domain. This assumption ensures that the initial point of the curve follows all necessary constraints, enabling us to determine how the descent flow advances within the permissible domain. 
    
    On the other hand, we  observe that the derivative of $f(\gamma(t))$ %at $t = 0$ 
    is given by:  
    \begin{align}\label{eq:DFT}
    \begin{array}{l}
    %\left.
    \frac{d f(\gamma(t))}{dt}\\
    %\right|_{t=0}  \\
    = %\left.
    \sum_{k=1}^{r} \frac{\partial f}{\partial p_k} \frac{d p_k}{dt} 
    + \sum_{k=1}^{r} \operatorname{tr}\left(\left[\frac{\partial f}{\partial U_k^\mathfrak{R}}\right]^\top \frac{d U_k^\mathfrak{R}}{dt} \right) 
    + \sum_{k=1}^{r} \operatorname{tr}\left(\left[\frac{\partial f}{\partial U_k^\mathfrak{I}}\right]^\top \frac{d U_k^\mathfrak{I}}{dt} \right),%\right|_{t=0}  ,
    \end{array}
    \end{align}
    where the superscript ``$\top$" denotes the transpose of the matrix.

    To derive the descent flow from \eqref{eq:DFT}, we update the tuple $(\mathbf{p}(t), U^{(1)}(t), \ldots, U^{(r)}(t))$ along the trajectory defined by the Euclidean derivative of \eqref{eq:par2}. However, since the optimization problem in \eqref{eq:optimization} is constrained by the unitary matrix structure and the probability simplex, updates must remain within these domains. To address this, we first consider the fundamental problem:  
    \begin{subequations}\label{eq:fundamental}
    \begin{align}
    \mbox{Minimize} & \quad g(U), \label{eq:fundamental1} \\
    \mbox{subject to} & \quad U \in  \mathcal{S}_{n}.\label{eq:fundamental2} 
    \end{align}
    \end{subequations}
    
    The gradient of $f(U)$ is expressed as:
    \[
    \nabla g(U) = \frac{\partial g(U)}{\partial U_k^\mathfrak{R}} + i \frac{\partial g(U)}{\partial U_k^\mathfrak{I}},
    \]
    and the steepest descent direction at $U \in \mathcal{S}_n$ is determined using the real-valued inner product and norm:
    \begin{eqnarray}\label{eq:inner}
    \langle A, B \rangle_{r} &=& \operatorname{Re}(\operatorname{tr}(A^* B)),\\
    \| A \|_{r} &=& \sqrt{\langle A, A \rangle_{r}},
    \end{eqnarray}
    where $A$ and $B$ are $n \times n$ complex matrices and $\operatorname{Re}(\cdot)$ denotes the real part of a complex number.
    Like \eqref{eq:DFT}, we see that the steepest descent direction 
     starting from a point  $U\in\mathcal{S}_n$ is determined by
    \begin{equation}\label{eq:steep}
    \xi_U = \underset{\xi \in T_U\mathcal{S}_n, \|\xi\|_U = 1}{\operatorname{argmin}} \langle \nabla g(U), \xi \rangle_U 
    = \frac{-\mathrm{Proj}_{T_U\mathcal{S}_n} (\nabla g(U))}{\|\mathrm{Proj}_{T_U\mathcal{S}_n} (\nabla g(U))\|_U},
    \end{equation}
    where $\mathrm{Proj}_{T_U\mathcal{S}_n} (\nabla f(U))$ is the projection of $\nabla f(U)$ onto the tangent space $T_U\mathcal{S}_n$.

    Recall that the tangent space $T_U \mathcal{S}_n$ attached to the point $U$ is characterized as:
    \begin{eqnarray}\label{eq:tangent}
    T_U \mathcal{S}_n &=& \{Z \in \mathbb{C}^{n \times n} : U^* Z + Z^* U = 0 \} \nonumber \\
    &=& U \mathcal{H}_n^\perp, \nonumber
    \end{eqnarray}
    where $\mathcal{H}_n^\perp$ is the space of skew-Hermitian matrices (see \cite{Edelman1998} for further details). Using this, the steepest descent flow on $\mathcal{S}_n$ is given by:
    \begin{equation}\label{eq:sdf}
    \frac{dU(t)}{dt} = -U \, \mathrm{skew}(U^* \nabla g(U)),
    \end{equation}
    where $\mathrm{skew}(A) = \frac{1}{2}(A - A^*)$ represents the skew-Hermitian part of the matrix $A$.
    
    Similarly, let $\mathbf{p}(t) = [p_i(t)]$ with $\mathbf{p}(0) > 0$ evolve on the probability simplex $\Delta^{r-1}$. To preserve the trace-one property and ensure reducing the objective value in~\eqref{eq:objective}, we enforce:
    \begin{equation}\label{eq:sumto1}
    \sum_{i=1}^r \frac{d p_i(t)}{dt} = 0 \quad \text{for all } t \geq 0.
    \end{equation}
    Combining these results, we obtain the modified continuous-time descent system for minimizing $f$ given in~\eqref{eq:optimization}:
    \begin{equation} \label{eq:dyflow}
    \left \{
    \begin{aligned}
        \frac{dU_k}{dt}(t) &= -U_k \, \mathrm{skew}\left(U_k^* \frac{\partial f}{\partial U_k}\right), \\ 
        \frac{dp_k}{dt}(t) &= -\frac{\partial f}{\partial p_k} + \frac{1}{r} \sum_{\ell=1}^r \frac{\partial f}{\partial p_\ell},
    \end{aligned}
    \right .
    \end{equation}
    for $k = 1, \ldots, r$. 
    
    By construction, we observe that  
    \[
    \frac{d \| U_k \|^2_F}{d t} = 2 \, \mathrm{Re} \left( \left\langle U_k(t), \frac{d U_k(t)}{d t} \right\rangle \right) = 0,
    \]  
    where the last equality follows from the fact that $ \mathrm{Re} \left( \left\langle A, B \right\rangle \right) = 0$ for any Hermitian matrix $A$ and skew-Hermitian matrix $B$.  This implies that the trajectory $U_k(t)$ remains bounded in its Frobenius norm for all $t$ and $k=1,\ldots,r$. On the other hand, note that the optimization problem described in \eqref{eq:optimization} imposes two constraints on the coefficients $p_k$: they must be nonnegative and they must sum to one. In contrast, the dynamics described in \eqref{eq:dyflow} inherently preserves only the property of summing to one. However, ensuring non-negativity is a manageable challenge. One can address this issue by using standard techniques for solving ordinary differential equations. A well-known and effective approach is to utilize the event detection feature available in MATLAB's built-in ``ode" solver. This feature allows us to define a custom event function that monitors the system during integration. Specifically, we construct the event function to detect the exact moment, denoted $\hat{t}$, when any coefficient $p_{\hat{k}}(\hat{t})$, for some index $\hat{k}$, becomes zero. Identifying this critical moment is important because continuing the integration beyond this moment would violate the non-negativity constraint. Furthermore, once $p_{\hat{k}}(\hat{t})$ reaches zero, its contribution to the term $p_{\hat{k}}(\hat{t}) U_{\hat{k}}(\hat{t}) \rho U_{\hat{k}}(\hat{t})^*$ becomes redundant in evaluating the optimal value. At this point, we pause the iteration and restart the process from the current state, excluding the coefficient $p_{\hat{k}}$ and its associated unitary matrix $U_{\hat{k}}$ from further iterations. This approach ensures the solution's boundedness and enhances computational efficiency by adaptively reducing the problem's dimensionality. We outline the complete procedure in Algorithm~\ref{alg:modified_descent}.

\begin{algorithm}[ht]
\caption{Modified Continuous-Time Descent Flow}
\label{alg:modified_descent}
\begin{algorithmic}[1]
\State \textbf{Input:} Initial values $\{p_k(0), U_k(0)\}$
\State \textbf{Output:} Optimized values $\{p_k, U_k\}$

\While{Optimization has not converged}
    \State Use an ODE solver to integrate \eqref{eq:dyflow} with initial values $\{p_k(0), U_k(0)\}$.
    \If{there exists an index $\hat{k}$ and time $\hat{t}$ such that $p_{\hat{k}}(\hat{t}) = 0$}
        \State Remove the component $(p_{\hat{k}}, U_{\hat{k}})$ from further optimization.
        \State Restart integration with updated initial values $\{p_k(0), U_k(0)\} \gets \{p_k(\hat{t}), U_k(\hat{t})\}$ for $k \neq \hat{k}$.
    \EndIf
\EndWhile
\end{algorithmic}
\end{algorithm}

    \section{Convergent Analysis}
    When solving~\eqref{eq:optimization} using the continuous-time differential system~\eqref{eq:dyflow}, it is essential to analyze the convergence of its dynamic behavior.
    To this end, we first note that using the Frobenius norm, the objective function satisfies $f(\mathbf{p}, U_1, \ldots, U_k) \geq 0$ for all $(\mathbf{p}, U_1, \ldots, U_k)$.
    From~\eqref{eq:dyflow}, let $\gamma(t) = (\mathbf{p}(t), U_1(t), \ldots, U_k(t))$. Below, we demonstrate the diminishing behavior of the objective function $f$ by using Algorithm~\ref{alg:modified_descent}.
    \begin{theorem}\label{thm:conv}
    Let $\gamma(t) = (\mathbf{p}(t), U_1(t), \ldots, U_k(t))$ represent the flow defined by~\eqref{eq:dyflow}. Then, the objective value $f(\gamma(t))$ in~\eqref{eq:optimization} does not increase over time along this trajectory $\gamma(t)$.
     \end{theorem}
    \begin{proof}
    Let us explain the result by first analyzing the inner product:  
    \begin{eqnarray}\label{eq:normU}  
    &&\operatorname{tr}\left(\left[\frac{\partial f}{\partial U_k^\mathfrak{R}
    }\right]^\top \frac{d U_k^\mathfrak{R}
    }{dt} \right) 
    +\operatorname{tr}\left(\left[\frac{\partial f}{\partial U_k^\mathfrak{I}
    }\right]^\top \frac{d U_k^\mathfrak{I}
    }{dt} \right)\nonumber\\
    % \left\langle \frac{\partial f}{\partial U_k}, \frac{dU_k}{dt} \right\rangle_r          
    &&= \operatorname{Re}\left(\left\langle \frac{\partial f}{\partial U_k}, -
     U_k \frac{\left( U_k^* 
    \frac{\partial f}{\partial U_k} -  
    {\frac{\partial f}{\partial U_k}}^* U_k 
    \right)}{2} \right\rangle\right ) \nonumber\\  
    &&=  \frac{1}{2}
    \left[
    -  \left\langle \frac{\partial f}{\partial U_k}, \frac{\partial f}{\partial U_k} \right\rangle
    + 
    \operatorname{Re} \left(
    \left\langle U_k^* \frac{\partial f}{\partial U_k}, {\frac{\partial f}{\partial U_k}}^* U_k  \right\rangle\right)  
    \right] \nonumber\\  
    &&\leq  
    \frac{1}{2}
    \left[
    -  \left\langle \frac{\partial f}{\partial U_k}, \frac{\partial f}{\partial U_k} \right\rangle
    + \left\| U_k^* \frac{\partial f}{\partial U_k} \right\|_F 
    \left
    \| {\frac{\partial f}{\partial U_k}}^* U_k \right\|_F  
    \right] = 0.  
    \end{eqnarray}  
    Here, the inequality follows from the Cauchy-Schwarz inequality, ensuring that the result is non-positive.
    
    Second, we observe that for the coefficients $p_k$, the corresponding inner product satisfies  
    \begin{eqnarray}\label{eq:5}  
    \sum_{k=1}^r \left\langle \frac{\partial f}{\partial p_k}, \frac{d p_k}{dt} \right\rangle  
      &=&  
      \sum_{k=1}^r \left\langle \frac{\partial f}{\partial p_k},  
    \frac{\partial f}{\partial p_k} + \frac{1}{r} \sum_{\ell=1}^r \frac{\partial f}{\partial p_\ell}  
      \right\rangle \nonumber\\  
      &=& -\left(\sum_{\ell=1}^r 
      \left\|\frac{\partial f}{\partial p_\ell}\right\|_F^2  - \frac{1}{r} \sum_{\ell=1}^r \frac{\partial f}{\partial p_\ell} \sum_{m=1}^r \frac{\partial f}{\partial p_m} \right) \nonumber\\  
      &\leq& -\left(1 - \frac{1}{r}\right)\sum_{\ell=1}^r \left\|\frac{\partial f}{\partial p_\ell}\right\|_F^2 
      \leq 0,  
    \end{eqnarray}  
    where the final inequality follows from the Cauchy–Schwarz inequality. 
    Finally, by computing the derivative of $f$ along the trajectory of the solution $\gamma(t)$, we find that:
    \begin{equation*}
    \frac{d}{dt} f(\gamma(t)) \leq 0, 
    \end{equation*}
    by applying the results established in~\eqref{eq:DFT}, ~\eqref{eq:normU}, and~\eqref{eq:5} and completes the proof.
    % {\color{red}
    % On the other hand, this result also confirms the desired Lyapunov properties, ensuring that $f$ decreases (or remains constant) along the trajectory, validating the global convergence of the dynamic system.}
    \end{proof}

    \begin{corollary}
    Let $\gamma(t) = (\mathbf{p}(t), U_1(t), \ldots, U_k(t))$ represent the flow defined by~\eqref{eq:dyflow}, and let $f(\gamma(t))$ denote the objective value in~\eqref{eq:optimization}. Then $\frac{df(\gamma(t))}{dt} = 0$ if and only if $\frac{d\gamma(t)}{dt} = 0$.
    \end{corollary}
    \begin{proof}
    From~\eqref{eq:DFT}, we observe that $\frac{d\gamma(t)}{dt} = 0$ implies $\frac{df(\gamma(t))}{dt} = 0$. 
    
    Furthermore, from~\eqref{eq:normU} and~\eqref{eq:5}, we deduce that 
    $\frac{df(\gamma(t))}{dt} = 0$ implies 
    \begin{equation} \label{eq:eq1}
    \left\{
    \begin{aligned}
      & \operatorname{Re}\left(
      \left\langle
    \frac{\partial f}{\partial U_k}, \frac{d U_k}{dt}
      \right\rangle
      \right) = 0, \quad \text{for all } k=1, \ldots, r, \\ 
      &   \sum_{k=1}^r 
      \left\langle
    \frac{\partial f}{\partial p_k}, \frac{d p_k}{dt}
      \right\rangle = 0,
    \end{aligned}
    \right.
    \end{equation}
    
    Additionally, we observe that
    \begin{eqnarray}\label{eq:normdudt}  
    \left\langle \frac{d U_k}{dt},
    \frac{d U_k}{dt}
    \right\rangle            
    &=& 
    \left\langle -
     U_k \frac{\left( U_k^* 
    \frac{\partial f}{\partial U_k} -  
    {\frac{\partial f}{\partial U_k}}^* U_k 
    \right)}{2}, -
     U_k \frac{\left( U_k^* 
    \frac{\partial f}{\partial U_k} -  
    {\frac{\partial f}{\partial U_k}}^* U_k 
    \right)}{2} \right\rangle
    \nonumber\\  
    &=&
    \frac{1}{4}
    \left\langle 
     U_k^*
    \frac{\partial f}{\partial U_k} -  
    {\frac{\partial f}{\partial U_k}}^* U_k, 
    U_k^* 
    \frac{\partial f}{\partial U_k} -  
    {\frac{\partial f}{\partial U_k}}^* U_k 
     \right\rangle \nonumber\\
    &=&
    \frac{1}{2}
    \left(
    \left\|\frac{\partial f}{\partial U_k}
    \right\|_F^2 - \operatorname{Re}\left(\left\langle 
    \frac{\partial f}{\partial U_k}, U_k 
    {\frac{\partial f}{\partial U_k}}^* U_k  
     \right\rangle
    \right)
    \right)
     \nonumber\\ 
     &=& - \operatorname{Re}
     \left(\left\langle 
    \frac{\partial f}{\partial U_k}, 
    \frac{dU_k}{dt}   
     \right\rangle
    \right) = 0,
    \end{eqnarray}
    if $\frac{df(\gamma(t))}{dt} = 0$.
    Finally, from~\eqref{eq:5}, we also have 
    \[
    \sum_{k=1}^r 
      \left\langle
    \frac{\partial f}{\partial p_k}, \frac{d p_k}{dt}
      \right\rangle = 0,
    \]
    indicating that $\left\|\frac{\partial f}{\partial p_\ell}\right\|_F = 0$, i.e., $\frac{\partial f}{\partial p_\ell} = 0$ for $k=1, \ldots, r$, which completes the proof.
    \end{proof}

    %
    % Let $(\widehat{\mathbf{p}}, \widehat{U_1}, \ldots, \widehat{U_k})$ denote the equilibrium point that satisfies
    % \begin{equation} \label{eq:dyflow2}
    % \left \{
    % \begin{aligned}
    %   & -U_k \, \mathrm{skew}\left(U_k^* \frac{\partial f}{\partial U_k}\right)
    %   \Big|_{ (\mathbf{p}, U_1, \ldots, U_r) = (\widehat{\mathbf{p}}, \widehat{U_1},\ldots, \widehat{U_r}) } = 0, \\ 
    %   & -\frac{\partial f}{\partial p_k} + \frac{1}{r} \sum_{k=1}^r \frac{\partial f}{\partial p_k}
    %   \Big|_{ (\mathbf{p}, U_1, \ldots, U_k) = (\widehat{\mathbf{p}}, \widehat{U_1}, \widehat{U_k}) } = 0,
    % \end{aligned}
    % \right.
    % \end{equation}
    % for $k = 1, \ldots, r$.
    % 
    %

    {\color{blue}
    \begin{lemma}\label{thm:mainasymp}
    Let $\phi:\mathbb{R}^n \to \mathbb{R}$ be a function satisfying $\phi(\mathbf{x}) \geq 0$. Furthermore, suppose that its derivative along any trajectory $\mathbf{x}(t)$ governed by
    \begin{equation}\label{eq:dyflowh} \frac{d\mathbf{x}(t)}{dt} = h(\mathbf{x}(t)) \end{equation}
    satisfies $\frac{d}{dt} \phi(\mathbf{x}(t)) \leq 0$. 
    Furthermore, suppose that at some particular time $\hat{t}$, the condition  
    \[
    \left.\frac{d\phi(\mathbf{x}(t))}{dt}\right|_{\hat{t}} = 0
    \]  
    holds if and only if  
    \[
    \left.\frac{d\mathbf{x}(t)}{dt}\right|_{\hat{t}} = 0,
    \]  
    or equivalently, $h(\mathbf{x}(\hat{t})) = 0$, and assume that for all $t \geq 0$, the trajectory $\mathbf{x}(t) \subset \mathbb{R}^n$ is compact and the equilibrium points of the dynamical system~\eqref{eq:dyflowh} are isolated.
     Then, for any initial condition $\mathbf{x}_0$, the solution $\mathbf{x}(t)$ converges to one of these equilibrium points, denoted by $\\hat{\mathbf{x}}$, i.e., 
    \begin{equation}
        \lim_{t \to \infty} \mathbf{x}(t) = \hat{\mathbf{x}}.
    \end{equation}
    Equivalently, the $\omega$-limit set of $\mathbf{x}(t)$ consists solely of the point $\mathbf{x}^*$.  
    \end{lemma}
    
    \begin{proof}
     
    First, we show that the $\omega$-limit set is contained in $\left\{\mathbf{x}: 
    \frac{d}{dt} \phi(\mathbf{x}(t))
    % \nabla \phi(\mathbf{x})\cdot f(\mathbf{x}) 
    = 0\right\}$. By the definition of the $\omega$-limit point, there exists a strictly increasing sequence $t_k \to \infty$ and a particular point $\hat{\mathbf{x}}\in\mathbb{R}^{n}$ for which
    \begin{eqnarray*}
        \mathbf{x}(t_k) \to \hat{\mathbf{x}} \quad \text{as}\ k \to \infty.
    \end{eqnarray*}
    By assumption, $\phi(\mathbf{x}(t))$ is a continuous and non-increasing function along the trajectory $\mathbf{x}(t)$, which is  defined on a compact set. Consequently, for any strictly increasing sequence $s_k \to \infty$, the sequence $\left\{\phi(\mathbf{x}(s_k))\right\}$ is non-increasing. Since $\left\{\phi(\mathbf{x}(s_k))\right\}$ is non-increasing and bounded (due to the compactness of the set and continuity of $\left\{\phi(\mathbf{x}(t))\right\}$), it converges to a limit, denoted by
    \begin{equation*}
        \lim_{k \to \infty} \phi(\mathbf{x}(s_k)) = \phi^*.
    \end{equation*}
    Because $\mathbf{x}(t_k) \to x_\infty$ and $\phi$ is continuous, we have
    \begin{eqnarray*}
        \lim_{k \to \infty} \phi\bigl(\mathbf{x}(t_k)\bigr) = \phi(\hat{\mathbf{x}}).
    \end{eqnarray*}
    
    Without loss of generality, we assume that for all $k \geq 1$, $t_k \leq s_k$. Since $\{\phi(\mathbf{x}(t))\}$ is a non-increasing sequence for all $t \geq 0$, it follows that $\phi(t_k) \geq \phi(s_k)$ for all $k \geq 1$. Consequently, we have $\phi(\hat{\mathbf{x}}) \geq \phi^*$. Similarly, we can select specific subsequences $\{t_{k_j}\}$ and $\{s_{k_j}\}$ such that $t_{k_j} \geq s_{k_j}$ and show that $\phi^* \geq \phi(\hat{\mathbf{x}})$. Together, these results imply $\phi(\hat{\mathbf{x}}) = \phi^*$. Thus, for any strictly increasing sequence $\{r_k\}$ with $r_k \to \infty$ as $k \to \infty$, we have $\phi(\mathbf{x}(r_k)) \to \phi^*$ as $k \to \infty$. Therefore, $\phi(\mathbf{x}(t)) \to \phi^*$ as $t \to \infty$. Hence, $\phi^*$ must be a local minimum of $\mathbf{x}(t)$.
    
    On another note, we have 
    \[
    \frac{d}{dt} \phi(\mathbf{x}(t)) = \langle \nabla \phi(\mathbf{x}(t)), h(\mathbf{x}(t)) \rangle \leq 0 \quad \text{for all } t \geq 0.
    \]
    Furthermore, 
    \begin{equation*}
    \lim_{k \to \infty} \langle \nabla \phi(\mathbf{x}(t_k)), h(\mathbf{x}(t_k)) \rangle = \langle \nabla \phi(\hat{\mathbf{x}}), h(\hat{\mathbf{x}}) \rangle = 0,
    \end{equation*}
    since $\phi(\hat{\mathbf{x}}) = \phi^*$ is a local minimum. This implies that $h(\hat{\mathbf{x}}) = 0$, i.e., $\hat{\mathbf{x}}$ is an equilibrium of the dynamical system given by~\eqref{eq:dyflowh}.
    
    We will then demonstrate that $\mathbf{x}(t)$ converges to $\hat{\mathbf{x}}$ as $t \to \infty$. Since the flow $\mathbf{x}(t)$ is continuous and bounded for all $t \geq 0$, suppose that $\mathbf{x}(t)$ has two distinct $\omega$-limit points, $\hat{\mathbf{x}}_1$ and $\hat{\mathbf{x}}_2$, with $\hat{\mathbf{x}}_1 < \hat{\mathbf{x}}_2$. Then, there exist two sequences $\{p_n\}$ and $\{q_n\}$ such that $\mathbf{x}(p_n) \to \hat{\mathbf{x}}_1$ and $\mathbf{x}(q_n) \to \hat{\mathbf{x}}_2$ as $n \to \infty$. Let $\mathbf{y}$ be a point in $(\hat{\mathbf{x}}_1, \hat{\mathbf{x}}_2)$. Then $\mathbf{y}$ must be an $\omega$-limit point. Otherwise, there exist two positive numbers $\epsilon$ and $t_\ell$ such that $|\mathbf{x}(t) - \mathbf{y}| > \epsilon$ for all $t \geq t_\ell$. This contradicts the fact that the flow between $(\hat{\mathbf{x}}_1, \hat{\mathbf{x}}_2)$ must be continuous. Therefore, this implies that every point in $[\hat{\mathbf{x}}_1, \hat{\mathbf{x}}_2]$ is an $\omega$-limit point and an equilibrium point, which contradicts our assumption that the equilibrium points are isolated. This completes the proof.
    \end{proof}

    Furthermore, our convergence analysis counts on the following result for geometrically isolated solutions of a generic polynomial system~\cite[Theorem 7.1.1]{Sommese05}.
    \begin{lemma}\label{lem:isolated}
    Let $P(\mathbf{z}; \mathbf{q})$ be a system of polynomials with variables $\mathbf{z}\in \mathbb{C}^n$ and parameters $\mathbf{q}\in\mathbb{C}^m$.  Define $\mathcal{N}(\mathbf{q})$ as the number of geometrically isolated solutions satisfying the condition: 
    \begin{equation*}
    \mathcal{N}(\mathbf{q}) : = \# \left\{\mathbf{z} \in\mathbb{C}^n \left |  P(\mathbf{z}; \mathbf{q}) = 0, \det\left (\frac{\partial P}{\partial \mathbf{z}} (\mathbf{z}; \mathbf{q}) \right) \neq 0\right.\right\}.
    \end{equation*}
    The following properties hold:
    \begin{enumerate}
    \item $\mathcal{N}(\mathbf{q})$ is finite and remains constant, denoted as $\mathcal{N}$, for almost all $\mathbf{q}\in \mathbb{C}^m$;
    
    \item For all $\mathbf{q}\in \mathbb{C}^m$, it follows that $\mathcal{N}(\mathbf{q}) \leq \mathcal{N}$;
    
    \item  The subset of $\mathbb{C}^m$ where $\mathcal{N}(\mathbf{q}) = \mathcal{N}$ is a Zariski open set. In other words, the exceptional subset of $\mathbf{q} \in \mathbb{C}^m$ where $\mathcal{N}(\mathbf{q}) <\mathcal{N}$ is an affine algebraic set contained within an algebraic set of dimension $n-1$.
    \end{enumerate}
    
    \end{lemma}
    Note that the set $\mathbb{R}^n$ is Zariski dense in $\mathbb{C}^n$~\cite{Dong2020}. Thus, the properties described above hold for almost all parameters $q \in \mathbb{R}^m$, although the number of isolated solutions of real value of the function varies and is no longer constant. Despite this imperfection, this result is sufficient for our purposes, as it establishes the necessary conditions for the subsequent discussion.

    By utilizing Lemma~\ref{lem:isolated} and Lemma~\ref{thm:mainasymp}, along with the established condition of the boundedness for this dynamical flow~\eqref{eq:dyflow}, we can demonstrate the following convergence property.  
    \begin{theorem}\label{thm:mainasymp0} 
    Let $\gamma(t) = (\mathbf{p}(t), U_1(t), \ldots, U_k(t))$ represent the flow defined in equation~\eqref{eq:dyflow}. Let $\gamma^*$ be an $\omega$-limit point of the flow $\gamma(t)$. Then, we have 
    \begin{equation} \label{eq:mainsaymp}
    \lim\limits_{t\rightarrow \infty} \gamma(t) = \gamma^*
    \end{equation}
    almost surely for any initial value $\gamma(0)$.
    \end{theorem}

    }

    \section{Numerical Experiments}
    In this section, we present three experiments demonstrating a decreasing trend of the objective function along the defined trajectory while addressing stability concerns. We have implemented our proposed method in MATLAB (version 2024b). For numerical integration, we utilized the ~\emph{ode15s} function with an absolute tolerance (Abstol) and a relative tolerance (Retol) set to $10^{-12}$, which allows the integrator to select the time step size adaptively. The program terminates once the objective function reaches $10^{-17}$, and {\color{red}we report the corresponding silhouettes. Despite the non-uniqueness of the approximated quantum channel, we demonstrate how our method can effectively approximate the original channel.}
    
    It is known that a mixed unitary quantum channel can have multiple ways of being decomposed. If a quantum channel admits the decomposition  
    \begin{eqnarray*}  
      \Phi(X) = \sum_{k=1}^r p_k U_k X U_k^{*},  
    \end{eqnarray*}  
    its corresponding Choi representation is given by  
    \begin{eqnarray*}  
        C(\Phi) = \sum_{k=1}^r \mathrm{vec}(\sqrt{p_k} U_k) \mathrm{vec}(\sqrt{p_k} U_k)^{*}.  
    \end{eqnarray*}  
    These two representations are equivalent~\cite{choi75, Watrous2018}. Thus, in our subsequent discussion, we use the Choi representation to determine whether different decompositions correspond to the same quantum channel.

    {\bf Example 1}
    This example illustrates the effectiveness of our proposed method through two experiments. 
    The first experiment tackles the following optimization problem~\eqref{eq:optimization}
    % \begin{eqnarray}\label{eq:example1single}
    %      &&\min_{\left\{p_k, U_k\right\}_{k=1}^N}
    %        \Bigl\|\sigma - \sum_{k}p_k\,U_k\,\rho\,U_k^*\Bigr\|_F^2, \\
    %     \text{subject to} 
    %     &&\quad \sum_{k=1}^N p_k = 1,\quad U_kU_k^* = I,\quad p_k \ge 0.
    % \end{eqnarray}
    The second experiment solves a similar problem but over multiple pairs of input and output quantum states:
   \begin{subequations}\label{eq:example1multi}
    \begin{align}
    \text{Minimize} & \quad \frac{1}{2} \sum_{j=1}^m \Bigl\|\sigma_j - \sum_{k = 1}^r p_k\,U_k\,\rho_j\,U_k^*\Bigr\|_F^2, \\
    \text{subject to} & \quad U_k \in \mathcal{S}_n, \quad k = 1, \ldots, r, \\
    & \quad \mathbf{p} \in \Delta^{r-1},
    \end{align}
    \end{subequations}

    Assume that we have a mixed unitary quantum channel $\Phi$ over $\mathbb{C}^{5\times 5}$ that is unknown to our program. We define $\Phi$ by randomly generating five unitary matrices $U_k$ and a set of probabilities $p_k$ (which sum up to 1), so that
    \begin{eqnarray*}
    \Phi(X) &=& \sum_{k=1}^5 p_k \, U_k \, X \, U_k^*.
    \end{eqnarray*}
    
    Next, we create the one-shot data. Let $\rho \in \mathbb{C}^{5 \times 5}$ be a randomly generated positive-definite Hermitian matrix, which we regard as the input quantum data. We then set $\sigma = \Phi(\rho)$ to be the corresponding output quantum data. {\color{blue} Without knowing the initial number $r = 5$, our numerical procedure sets the initial data with $R = 10$, which is twice the exact low rank, and randomly generates the initial data  
\begin{eqnarray*}  
\{p_k(0), U_k(0)\}_{k=0}^{R}  
\end{eqnarray*}  
corresponding to the prescribed structure in~\eqref{eq:optimization}.

%
%
%
% 02/28
%
%
 By employing Algorithm~\ref{alg:modified_descent}, we generate the flow
 \begin{eqnarray*}
 {p_k(t), U_k(t)}_{k=0}^{R}.
 \end{eqnarray*}
 Figure~\ref{fig:example1a} verifies that the constructed flow monotonically decreases the objective function, as rigorously established in Theorem~\ref{thm:conv}. The red circles in Figures~\ref{fig:example1a} and \ref{fig:example1b} mark critical moments when any $p_k(t)$ approaches zero, prompting a program restart to ensure the solution remains feasible. Importantly, Figure~\ref{fig:example1b} demonstrates that the sum of $p_k(t)$ consistently equals $1$ throughout all iterations, despite the observed minor fluctuations. These fluctuations arise due to numerical integration and rounding errors; however, their magnitudes remain relatively small and close to zero, ensuring the overall reliability of the approach.
    
    \begin{figure}[htbp]
        \centering
        % First subfigure
        \begin{subfigure}[t]{0.45\textwidth}
            \centering
            \includegraphics[width=\linewidth]{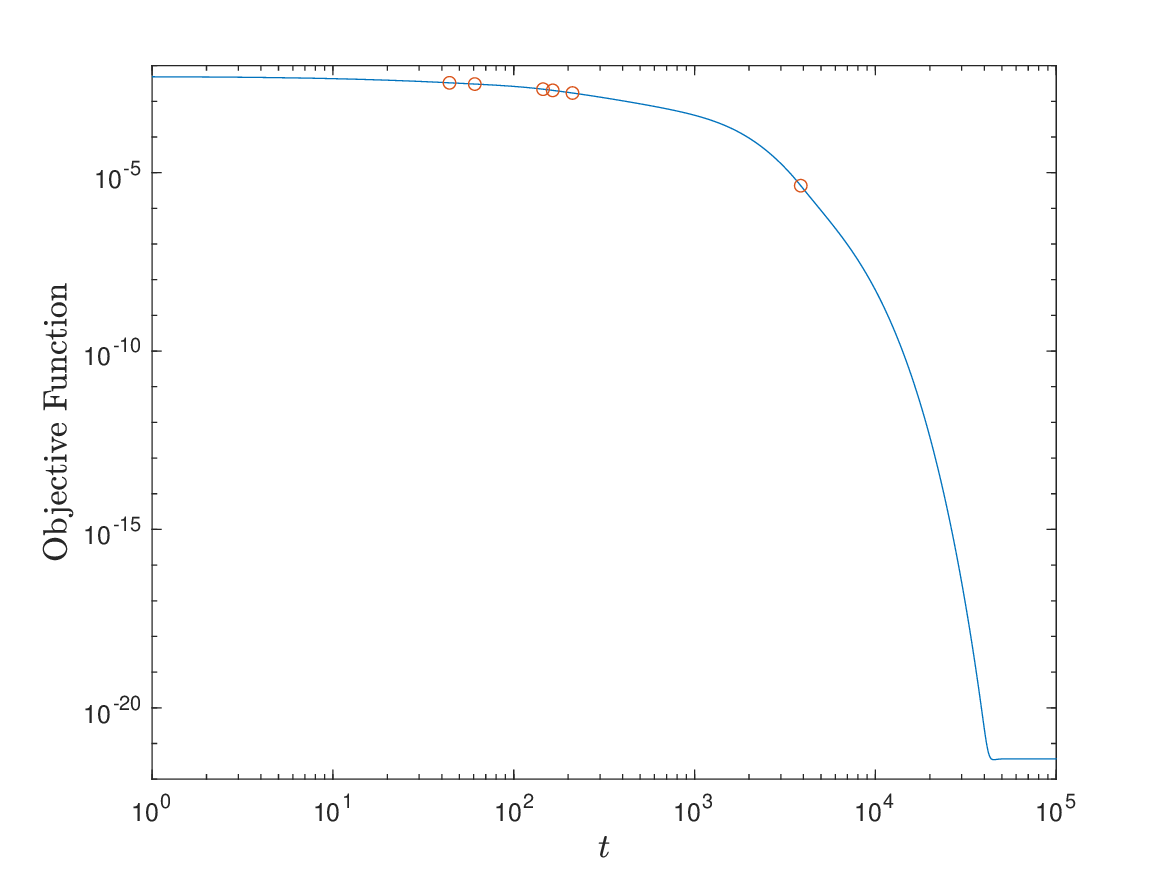}
            \caption{
            The evolution of the objective values of~\eqref{eq:optimization} 
            % $\frac{1}{2}\Bigl\|\sigma - \sum_{k}p_k(t)U_k(t)\rho U_k(t)^*\Bigr\|_F^2$
            }
            \label{fig:example1a}
        \end{subfigure}
        \hspace{1em}
        % Second subfigure
        \begin{subfigure}[t]{0.45\textwidth}
            \centering
            \includegraphics[width=\linewidth]{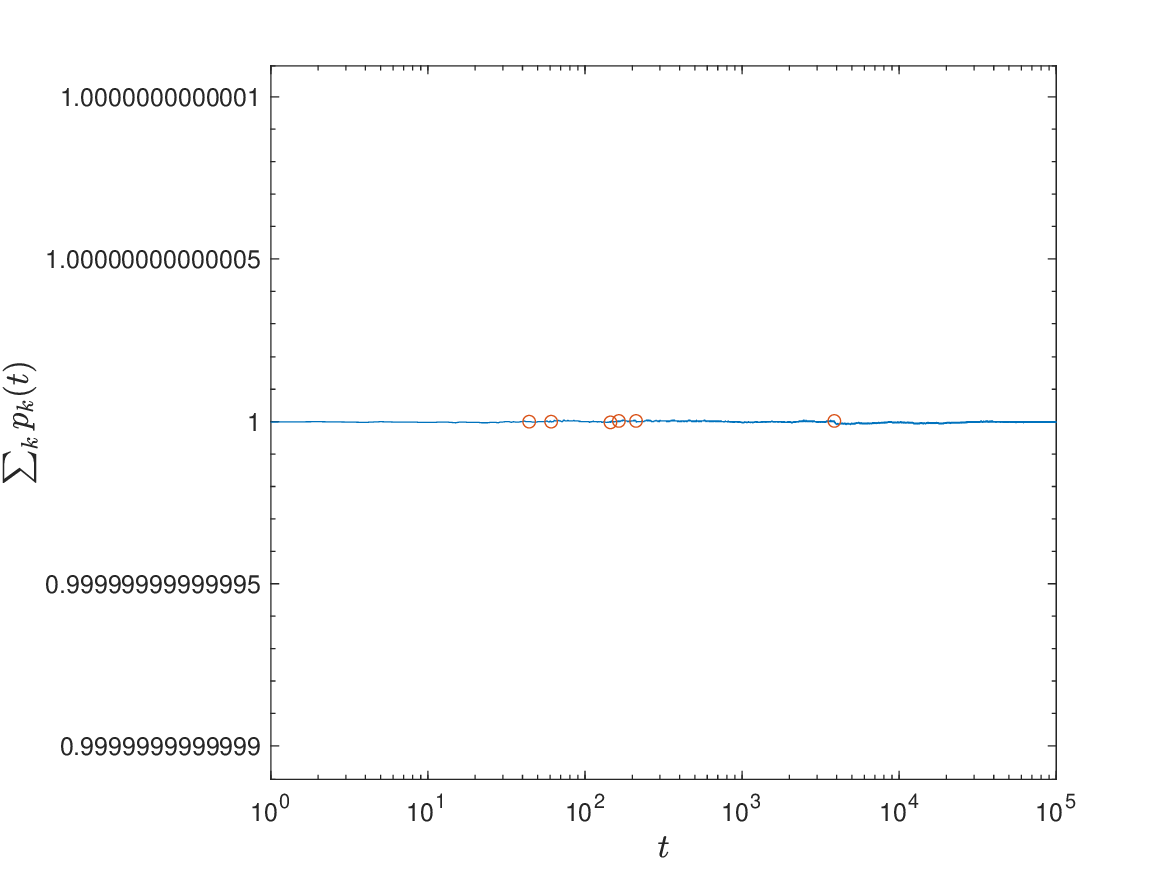}
            \caption{The evolution of $\sum_{k=1} p_k(t)$}
            \label{fig:example1b}
        \end{subfigure}
        \caption{Numerical results of solving~\eqref{eq:optimization}}
    \end{figure}
    
    The presence of five red circles in Figures~\ref{fig:example1a} and \ref{fig:example1b} indicates five restarts, which leads to a progressive reduction of the low-rank approximation to $5$. Since $\sigma$ is made up of precisely five quantum channels, the reduction of $R$ from $10$ to $5$ highlights the effectiveness of our method in identifying and eliminating redundant channels. Additionally, the objective function shows a substantial decrease, ultimately reaching $10^{-20}$ as $t$ nears $10^5$, highlighting the robustness and accuracy of our algorithm in optimizing the problem at hand.

Next, we investigate how multiple datasets can aid in recovering the initial channel $ \Phi $. It is important to note that the optimization problem in~\eqref{eq:optimization} may admit multiple solutions due to the degrees of freedom in $ U_k$ exceeding the amount of information the data provides. Based on the contributions of Choi and Jamio{\l}kowski, different quantum channels can correspond to distinct Choi matrix representations. To assess the robustness of our method, we run the optimization $20$ times with different initial guesses, then compute the difference between each resulting optimal channel $ \hat{\Phi} $ and the original channel $\Phi $ using their corresponding Choi matrices; specifically, we evaluate the deviation through the Frobenius norm $ \|C(\Phi) - C(\hat{\Phi})\|_F$. The results are shown in~\ref{fig:example1_1}, where the x-axis represents the difference between the computed optimal channel and the original quantum channel. Although our experiments still reveal that each run successfully reduces the objective function to $10^{-20}$, the resulting quantum channel still exhibits deviations from the true channel.

    \begin{figure}[htbp]
        \centering
        \includegraphics[width=0.5\linewidth]{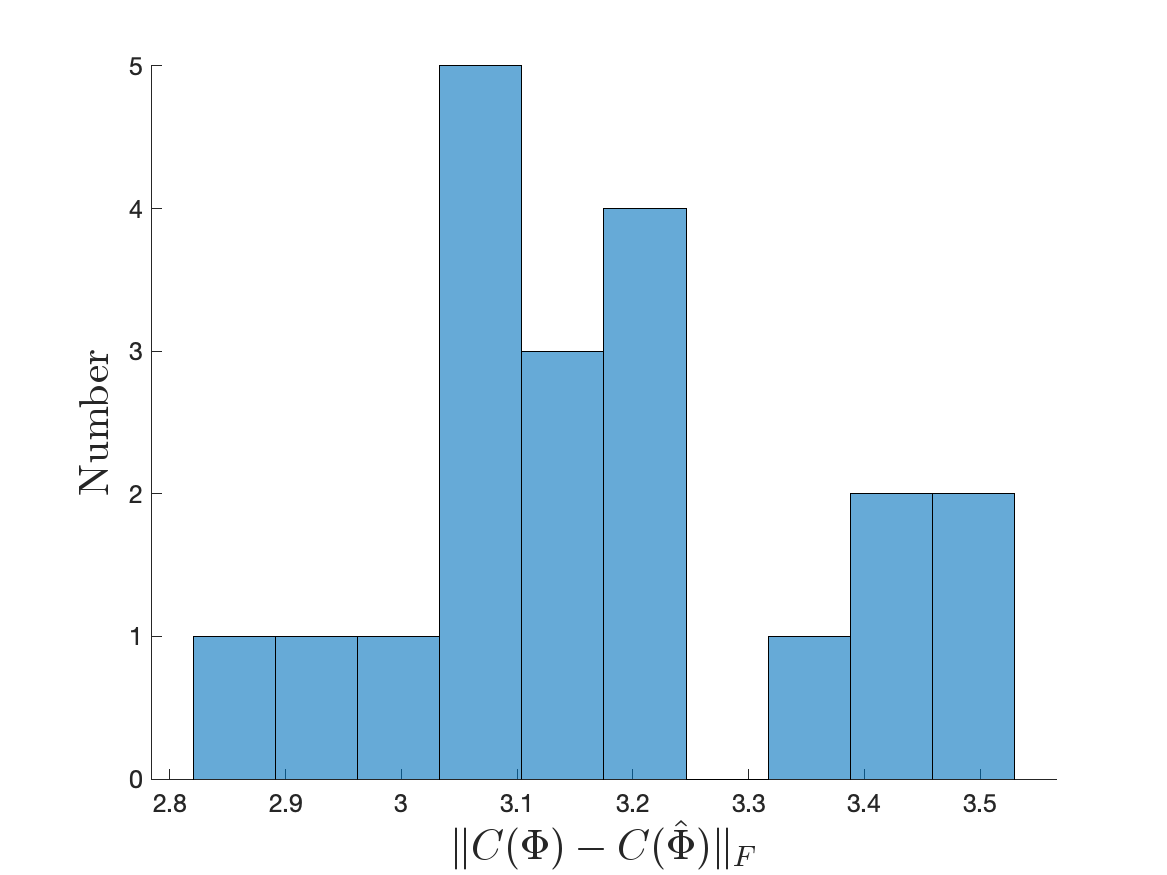}
        \caption{Values of $\|C(\Phi)-C(\hat{\Phi})\|_F$ collecting from 20 runs}
        \label{fig:example1_1}
    \end{figure}

     To address this issue, we aim to evaluate the impact of providing additional data pairs to enhance the likelihood of accurately approximating the original quantum channel. This is done by solving the optimization problem in~\eqref{eq:example1multi}. Establishing the dynamical system for \eqref{eq:example1multi} is similar to that given in  \eqref{eq:dyflow}, except that we sum over all data pairs; therefore, we omit the entire process for brevity. Specifically, we use the same setup but with $m=20$ pairs of input and output quantum states. Our primary objective is to verify that the proposed method continues to decrease the objective function while maintaining the sum-to-one property among the $p_k$. These features are demonstrated in Figure~\ref{fig:example1c} for the descent behavior and in Figure~\ref{fig:example1d} for the sum-to-one property, both of which confirm the applicability of our method to the multi-shot problem in~\eqref{eq:example1multi}.
    \begin{figure}[htbp]
        \centering
        % First subfigure
        \begin{subfigure}[t]{0.45\linewidth}
            \centering
            \includegraphics[width=\linewidth]{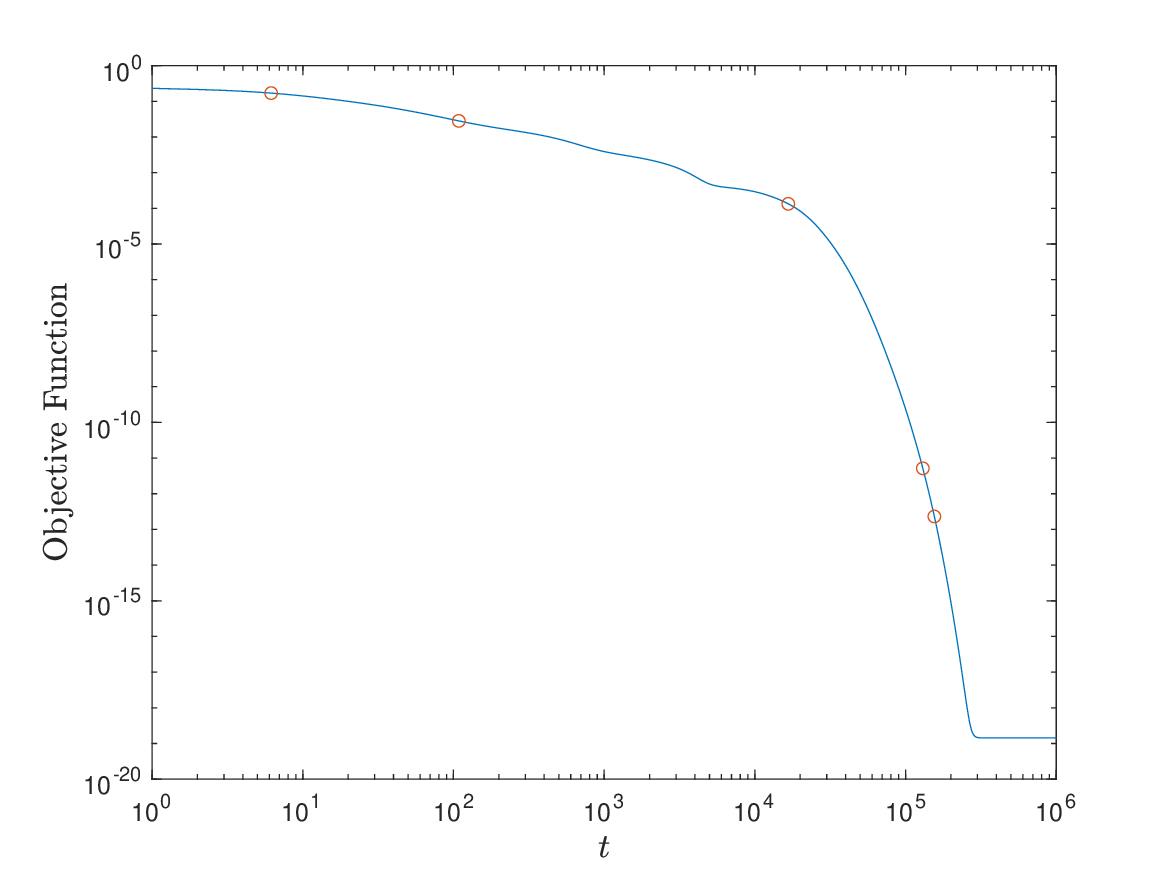}
            \caption{The evolution of the objective values of~\eqref{eq:example1multi} 
            % $\displaystyle\sum_{j=1}^{20}\Bigl\|\sigma_j - \sum_{k}p_k(t) U_k(t)\rho_j U_k(t)^*\Bigr\|_F^2$
            }
            \label{fig:example1c}
        \end{subfigure}
        \hfill
        % Second subfigure
        \begin{subfigure}[t]{0.45\linewidth}
            \centering
            \includegraphics[width=\linewidth]{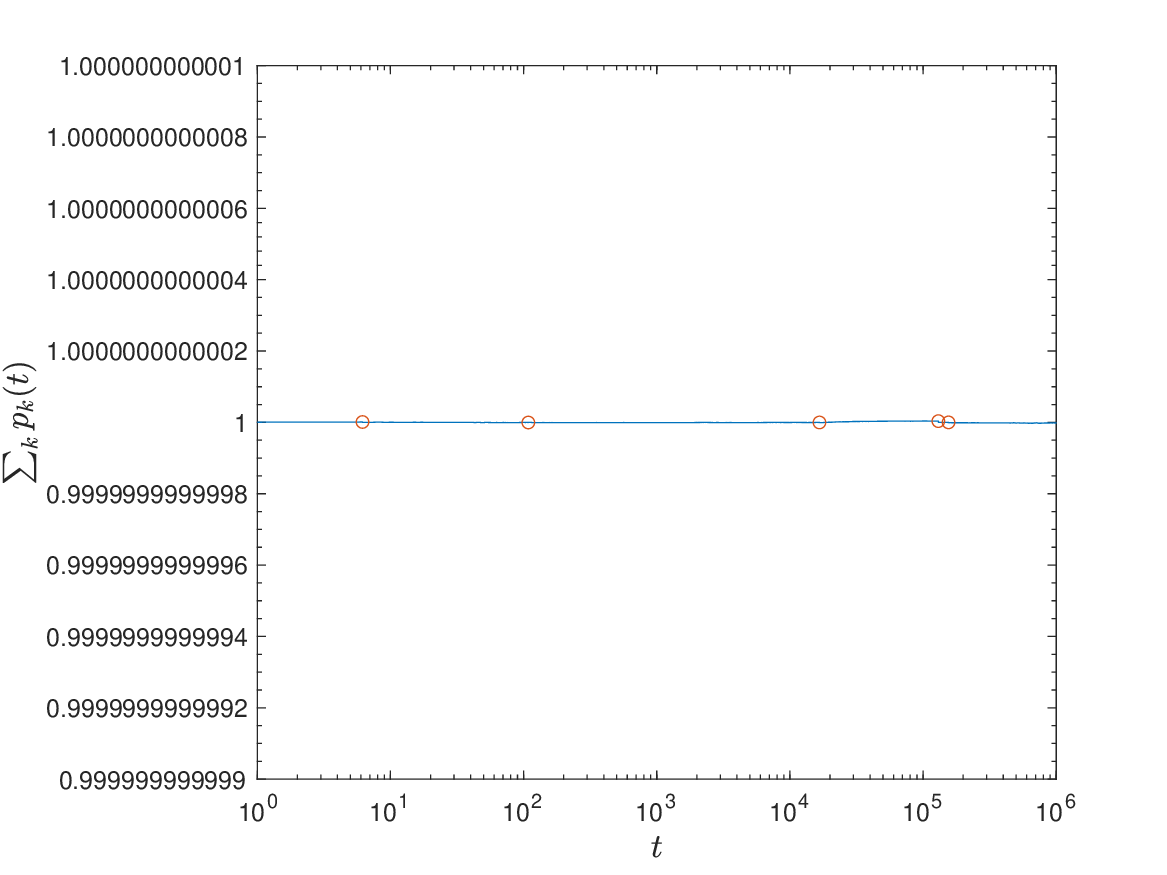}
            \caption{The evolution of $\sum_{k=1} p_k$}
            \label{fig:example1d}
        \end{subfigure}
    
        \caption{Numerical results of solving \eqref{eq:example1multi}, where the five circles label the restart occurrence.}
        % \label{fig:example1-cd}
    \end{figure}

Next, we evaluate whether additional data enhances the program's capacity to reconstruct the original quantum channel. 
{\color{red} 
We use 100 data points, i.e., $\{\sigma_j, \rho_j\}_{j=1}^{100}$, in each experiment, as expressed in \eqref{eq:example1multi}. This procedure is repeated 20 times. The distribution of the differences $\|C(\Phi) - C(\hat{\Phi})\|_F$ across the 20 runs is shown in Figure~\ref{fig:example1_2}.}
\begin{figure}[htbp]
    \centering
    \includegraphics[width=0.5\linewidth]{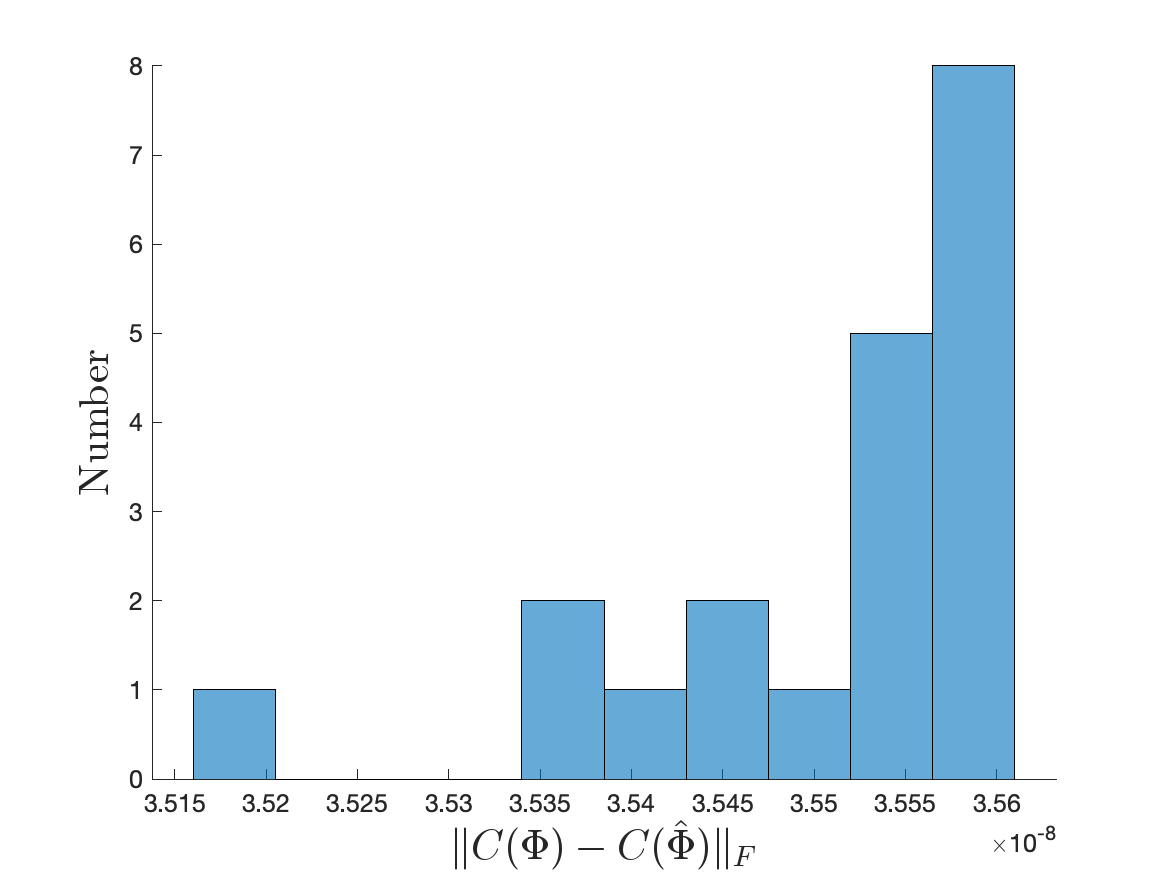}
    \caption{Distribution of $\|C(\Phi)-C(\hat{\Phi})\|_F$ collected from 20 runs.}
    \label{fig:example1_2}
\end{figure}  
In contrast to Figure~\ref{fig:example1_1}, where the x-axis covers a wider range, the x-axis in Figure~\ref{fig:example1_2} is centered around $3.5\times 10^{-8}$.The findings indicate that the reconstructed channels, derived from multiple data pairs, closely resemble the original channel in terms of their Choi matrix representations. Therefore, we conclude that our method can effectively reconstruct an unknown mixed unitary quantum channel when enough data pairs are available.

{\bf Example 2}
 The second example examines the depolarizing channel defined in~\eqref{eq:depolar}, where the parameter $p$ controls the strength of the noise. This channel is crucial for simulating errors in quantum information processing, while in quantum error correction, it aids in the design of codes that aim to mitigate the effects of quantum noise.The second example examines the depolarizing channel defined in~\eqref{eq:depolar}, where the parameter $p$ controls the strength of the noise. This channel is essential for simulating errors in quantum information processing and plays a significant role in quantum error correction, helping to design codes that can mitigate the impacts of quantum noise. We set $p = 0.9$ to characterize the channel and evaluate our proposed method by comparing the Choi matrix representation of the learned channel with that of the actual channel. We provide $20$ input data points and measure the corresponding outputs using the depolarizing channel to achieve this goal. The experiment is repeated $20$ times, as done in previous studies. We then plot the difference distribution between the Choi matrix representations of the actual and approximated channels.  

\begin{figure}[htbp]  
    \centering  
    \includegraphics[width=0.5\linewidth]{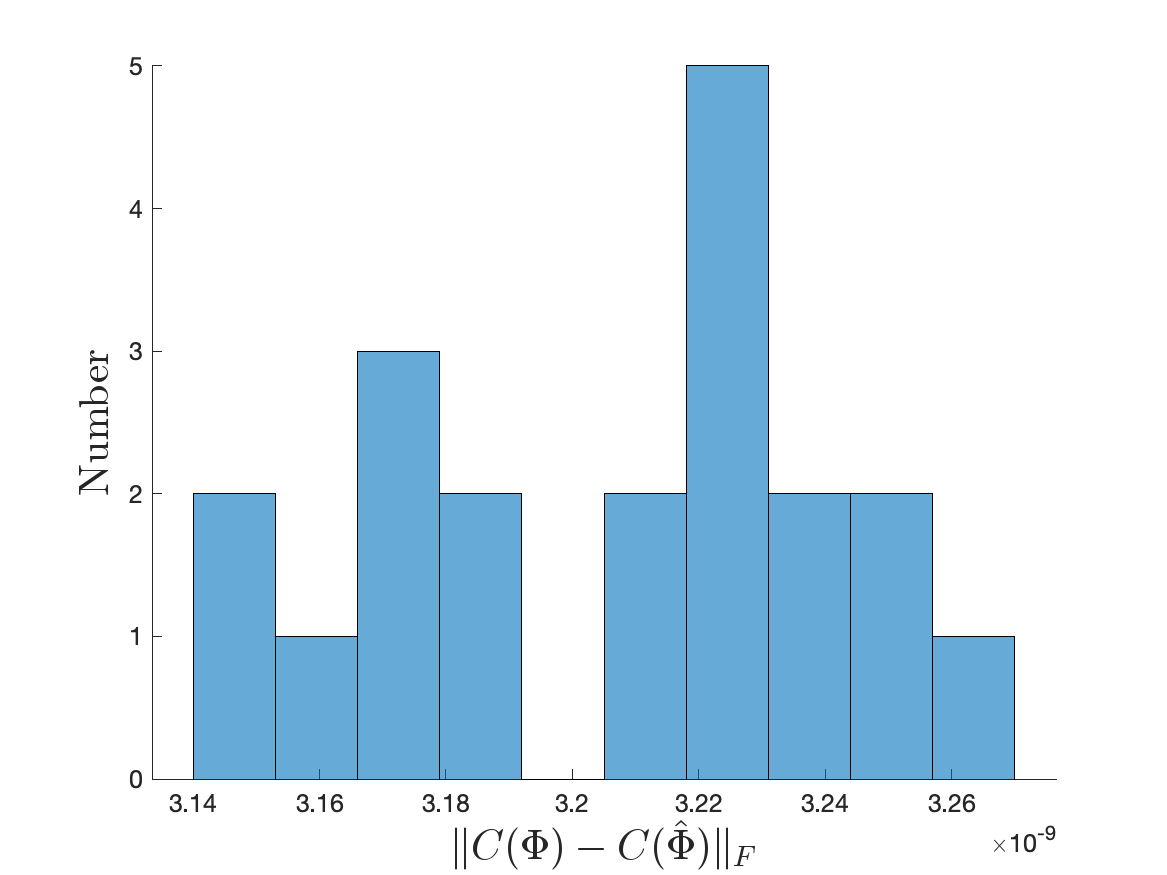}  
    \caption{Distribution of all $\|C(\Phi)-C(\hat{\Phi})\|_F$ collected from 20 runs.}  
    \label{fig:example1_3}  
\end{figure}  

As shown in Figure~\ref{fig:example1_3}, the difference in the Choi matrix representations between the approximated and actual depolarizing channels is concentrated around $3.2 \times 10^{-9}$ across the 20 runs. These results demonstrate that our method effectively identifies a quantum channel that closely approximates the unknown channel, resulting in minimal error in the Choi matrix representation.

    \section{Conclusion}
    This article proposes a descent flow approach to approximate an unknown unitary quantum channel. We formulate the problem as an optimization task on complex Stiefel manifolds and construct the flow using Wirtinger derivatives. Theoretically, we prove that the flow reduces the objective function while preserving the positivity and sum-to-one properties of the probability distribution $p_k$, which characterizes the approximated channel. Moreover, we show that the $\omega$-limit points obtained through our method are isolated and correspond to critical points of the objective function. These findings ensure the applicability of our approach in achieving an optimal solution. We validate the effectiveness and stability of our method numerically through experiments. Our approach is extended from single-shot to multi-shot data, and we evaluate the computed accuracy based on the Choi matrix representation of the quantum channel. The numerical results indicate that the proposed method can effectively approximate the unknown quantum channel by using multiple datasets.
    % In this article, we propose a descent flow approach to approximate an unknown unitary quantum channel using partial data. The problem is formulated as an optimization problem on complex Stiefel manifolds, and the flow is constructed using Wirtinger derivatives. Theoretically, we prove that the flow decreases the objective function while preserving the positivity and sum-to-one property of the probability set $p_k$, which defines the approximated channel. Furthermore, we employ the Lyapunov stability property to show that the $\omega$-limit points obtained by our method are isolated and, ultimately, critical points of the objective function. These results ensure the applicability of our method in reaching local minima of the defined objective function. We also present numerical experiments to demonstrate the effectiveness and stability of our approach. Additionally, we extend our method from single-shot data to multi-shot data and analyze the difference in the Choi matrix representation of the quantum channel. Our results indicate that the proposed method can approximate the unknown quantum channel more accurately when provided with a larger dataset.
    }

    \bibliographystyle{siamplain}
    \bibliography{quantum}
    \end{document}